\documentclass{amsart}
\usepackage{amssymb,verbatim,graphicx,tikz}

\usepackage{mathrsfs, amsthm, amssymb, amsmath}
\usepackage[margin=3cm]{geometry}
\usepackage{color, graphicx}
\usepackage{cancel}

\usepackage[english]{babel}
\usepackage{dsfont}
\usepackage[all]{xy}

\usepackage[colorlinks, citecolor=blue]{hyperref}

\newtheorem{lemma}{Lemma}

\newtheorem{corollary}{Corollary}
\newtheorem{theorem}{Theorem}
\newtheorem{non-theorem}{Non-Theorem}

\theoremstyle{remark}
\newtheorem{definition}{Definition}
\newtheorem{remark}{Remark}

\newcommand\bbC{\mathbb C}

\newcommand\bbP{\mathbb P}

\newcommand\bbR{\mathbb R}
\newcommand\bbS{\mathbb S}

\title{Eigenvalue resolution of self-adjoint matrices}
\author{Xuwen Zhu}
\address{Department of Mathematics, Stanford University}
\email{xuwenzhu@stanford.edu}
\date{\today}

\begin{document}

\begin{abstract}
Resolution of a compact group action in the sense described by Albin
and Melrose is applied to the conjugation action by the unitary
group on self-adjoint matrices. It is shown that the eigenvalues are
smooth on the resolved space and that the trivial bundle smoothly
decomposes into the direct sum of global one-dimensional
eigenspaces.
\end{abstract}

\maketitle

For a general compact Lie group $G$ acting on a smooth compact manifold with corners $M$, Albin and Melrose~\cite{MR2560748} showed that there is a canonical full resolution such that the group action lifts to the blown-up space $Y(M)$ to have a unique isotropy type. Under this conditon the result of Borel~\cite{borel2006compactifications} applies to show that the orbit space $G\backslash Y(M)$ is smooth. 

In this paper, we give an explicit construction of the resolution of the action of the unitary group on the space of self-adjoint matrices
\begin{equation*}
 S=S(n)=\{X \in
M_n(\mathbb{C})|X^*=X\}
\end{equation*}
with the unitary group $U(n)$ acting by conjugation: 
$$u \in
U(n), X \in S, u\cdot X:=uXu^{-1}.$$ 
The orbit of an element $X\in S$, denoted by $U(n)\cdot X$,
consists of the matrices with the same eigenvalues including multiplicities. For a matrix $X
\in S$ with $m$ distinct eigenvalues $\{\lambda_j\}_{j=1}^m$ with multiplicities $i_k, k = 1,2, ..,m$,
the isotropy group of $X$ is conjugate to a direct sum of smaller unitary groups:
\begin{equation*}
U(n)^X(:=\{u\in U(n)|u\cdot X=X\})\cong \oplus_{k=1}^{m}U(i_k).
\end{equation*}
The isotropy types are therefore parametrized by the partition
of $n$ into integers. Note here that the partition contains information about ordering, for example, the two partitions of $3$, $\{i_1=1, i_2=2\}$ and $\{i_1=2, i_2=1\}$, are not the same type. 

For $n>1$, the eigenvalues are not smooth functions on $S$, but are singular where the multiplicities change. Consider the trivial bundle over $S$,
$M:=S\times \bbC^n,$ the fiber of which can be decomposed into $n$ eigenspaces of the self-adjoint matrix at the base point. This decomposition is not unique at matrices with multiple eigenvalues, and the eigenspaces are not smooth at these base points. 
We
will show that, by doing iterative blow-ups, the singularities are resolved and the eigenvalues become smooth functions on the resolved space. Moreover, by doing a ``full'' blow up, the eigenspaces also become smooth.

Recall the lemma of group action resolution in \cite{MR2560748}:
\begin{lemma}[\cite{MR2560748}]\label{resolution}
A compact manifold (with corners), M, with a smooth, boundary
intersection free, action by a compact Lie group, G, has a canonical
full resolution, $Y(M)$, obtained by iterative blow-up of minimal
isotropy types.
\end{lemma}

In this paper we will discuss two kinds of blow ups, namely radial and projective blow up, which
give different results; projective blow up of a hypersurface is trivial but radial blow up produces a new boundary.
A resolution of $S$ involves the choice of blow up and which centers to blow up. In this paper, we will discuss
three kinds of resolutions:
\begin{definition}\label{def}
We define the following three resolutions of $S$:
\begin{enumerate}
\item Radial resolution $\hat S_{r}$: radial blow up of all singular stratums $\{\exists i \neq j, \ \lambda_{i}=\lambda_{j}\}$ in an order
compatible with inclusion of the conjugation class of the isotropy
group;
\item Projective resolution $\hat S_{p}$: projective blow up of all singular stratums in the same order as radial resolution;
\item Small resolution $\hat S_{s}$: radial blow up of a smaller set of centers $\cup_{1\leq i < j \leq n}\{\lambda_{i}=\lambda_{i+1}=\dots=\lambda_{j}\}$ with the order determined by complete inclusion. 
\end{enumerate}
\end{definition}

As pointed out in ~\cite{MR2560748}, projective blow up usually requires an extra step of reflection in the iterative scheme in order to obtain smoothness. We will show that, the radial resolution yields that the
trivial bundle $M$ decomposes into the direct sum of $n$ 1-dimensional
eigenspaces. By contrast, after projective resolution or small resolution, the
eigenvalues are smooth on the resolved space, and locally is a smooth decomposition into simple eigenspaces, but the trivial bundle doesn't split into global line bundles. 

\begin{remark}
In theory there is a fourth resolution by doing projective blow up of the smaller set of centers introduced in $\hat S_{s}$. This resolves eigenvalues but does not globally resolve eigenbundles, for the same reason as $\hat S_{s}$. Therefore for simplicity we do not include this resolution in our discussion below.
\end{remark}

To describe the different outcome of the three resolutions above, we recall the resolution in the sense of Albin and Melrose.

\begin{definition}[Eigenresolution]

By an eigenresolution of $S$, we mean a manifold with corners $\hat
S$, with a surjective smooth map $\beta:\hat S \rightarrow S$ such
that the self-adjoint matrices have a smooth (local) diagonalization
when lifted to $\hat S$. Eigenvalues then lift to $n$ smooth functions $f_{i} $ on $\hat S$, i.e.
for any $X\in \hat S$, $\beta(X)$ has eigenvalues $\{f_{i}(X)\}_{i=1}^{n}$.
\end{definition}
Note that in the definition we only require the diagonalization to exist locally. 
To encompass the information of global decomposition of eigenvectors, we introduce the full resolution below. 

\begin{definition}[Full eigenresolution]
A full eigenresolution is an eigenresolution with global
eigenbundles. The eigenvalues lift to $n$ smooth functions $f_i$ on $\hat
S$, and the trivial n-dimensional complex vector bundle on $\hat S$
is decomposed into $n$ smooth line bundles: 
$$\hat S \times\mathbb{C}^n=\bigoplus_{i=1}^n E_i$$ 
such that
\begin{equation*}
\beta(X) v_i=f_i(X)v_i, \forall \ v_i \in E_i(X),\forall \  X \in \hat S.
\end{equation*}
\end{definition}

We use the blow-up constructions introduced by Melrose in the book ~\cite[Chapter 5]{resolution} and show that we can obtain resolutions in this way and, in particular, a full resolution if we use radial blow-up.

\begin{theorem}\label{blowup}
 The three types of resolutions given in Definition~\ref{def}, namely, $\hat S_{r}, \hat S_{p}$, and $\hat S_{s}$, each yields an eigenresolution. Only the radial resolution $\hat S_{r}$
gives a full eigenresolution.
\end{theorem}
\begin{remark}
In particular, the blow down map $\beta: \hat S \rightarrow S$ is a diffeomorphism between the interior of $\hat S$ and the open dense subset of S consisting of the matrices with $n$-distinct eigenvalues.
\end{remark}

Related to the problem of resolving eigenvalues is the problem of desingularisation of polynomial roots. In~\cite{kurdyka2008hyperbolic}, generalizing Rellich's result on one-dimensional analytical families~\cite{rellich1937storungstheorie}, the perturbation theory of hyperbolic polynomials is discussed using Hironaka's resolution theory. It is applied to perturbation theory of normal operators and resonances, see for example~\cite{rainer2013perturbation} and~\cite{rauch1980perturbation}.

The idea of resolution has been used in many geometric problems. The abstract notion of a resolution structure on a manifold with corners is discussed in~\cite{baum1985cohomologie}. In~\cite{davis1978smooth} it is shown that for a general action the induced action on the set of boundary hypersurfaces can be appropriately resolved. The canonical resolution is presented in~\cite{duistermaat2000lie}, and the induced resolution of the orbit space is considered in~\cite{hassell1995analytic}. In~\cite{MR2560748}, an iterative procedure is shown to capture the simultaneous resolution of all isotropy types in a ``resolution structure'' consisting of equivariant iterated fibrations of the boundary faces, which is the procedure we will use in this paper.

\textbf{Acknowledgement.} I would like to thank Richard Melrose for suggesting this project and all the helpful discussions.  I am also grateful to the anonymous referees for their careful reading and comments, and the suggestion of the small resolution discussed in this paper.

\section{Proof of Theorem~\ref{blowup}}
The proof of Theorem~\ref{blowup} proceeds through induction on the dimension. We begin by discussing the first example which is the $2\times 2$ matrices. 
\begin{lemma}[2 $\times$ 2 case]\label{2dim}
For the $2\times 2$ self-adjoint matrices $S(2)$, the eigenvalues and eigenvectors are smooth except at multiples of the identity. After radial resolution, the singularities are resolved and the trivial 2-dim bundle splits into the direct sum of two line
bundles. The projective resolution also gives smooth eigenvalues, but
does not give two global line bundles.
\end{lemma}
\begin{remark}
Note in the 2 by 2 case, the radial resolution $\hat S_{r}$ and the small resolution $\hat S_{s}$ are the same.
\end{remark}
\begin{proof}
In this case 
$$S=S(2)=\{\left( \begin{array}{cc} a_{11} & z_{12}\\
\bar z_{12} & a_{22}
\end{array}
\right)|a_{ii}\in \mathbb{R},z_{12} \in \mathbb{C}\} \cong
\mathbb{R}^4.$$ The space S is isomorphic to the product of $\mathbb{R}$ and
the trace-free subspace
\begin{equation}
 S_0=\{\left( \begin{array}{cc} a_{11} & z_{12}\\
\bar z_{12} & a_{22}
\end{array}
\right)|a_{11}+a_{22}=0\},
\end{equation}
i.e. there is a bijective linear map:
\begin{equation}\label{equ:S}
  \begin{array}{lccc}
  \phi: & S &\rightarrow& S_0 \times \mathbb{R}\\
   & A=\left( \begin{array}{cc} a_{11} & z_{12}\\
   \bar z_{12} & a_{22}
\end{array}
\right) & \mapsto& (A_0:=A-\frac{1}{2}(a_{11}+a_{22})I, \frac{1}{2}(a_{11}+a_{22})).
 \end{array}
\end{equation}

The eigenvalues $\lambda_i$ and eigenvectors $v_i$ of $A$ are related to those of $A_0$ by 
$\lambda_i(A)=\lambda_i(A_0)+\frac{1}{2}tr(A)$, $v_i(A)=v_i(A_0),i=1,2$.
Therefore, we can restrict the discussion of resolution to the subspace $S_0$,
since the smoothness of eigenvalues and eigenvectors on the resolution of $S$ follows.

Let $z_{12}=c+di$. The space $S_0$ can be identified with
$\mathbb{R}^3=\{(a_{11},c,d)\}$. The eigenvalues of this matrix are: 
\begin{equation}
 \lambda_{\pm}=\pm \sqrt{a_{11}^2+c^2+d^2}.
\end{equation}
Hence the only singularity of the eigenvalues on
$S_0$ is at the point $a_{11}=c=d=0$ which represents the zero matrix.

Based on the resolution formula in \cite{resolution}, the
radial blow up can be realized as
\begin{equation}
\hat S_{0,r}=[S_0,\{0\}]=S^+N\{0\}\sqcup (S_0\setminus \{0\})\simeq \bbS^2
\times [0,\infty)_+
\end{equation}
where the front face $S^+N\{0\}\simeq \bbS^2$. Here the radial variable is
$$r=\sqrt{a_{11}^2+c^2+d^2}.$$ 
The blow-down map is
\begin{equation}
 \beta: [S_0,\{0\}]\rightarrow S_0, (r,\theta)\mapsto r\theta, r\in \mathbb{R}_+, \theta \in \bbS^2.
\end{equation}
The radial variable $r$ lifts to be smooth on the blown up space, therefore the two eigenvalues $\lambda_\pm=\pm r$ become smooth functions.

Now we consider the eigenvectors to the corresponding eigenvalues $\lambda_{\pm}$
\begin{equation}\label{eigenvecor}
 v_{\pm}= (c+di, \pm \sqrt{a_{11}^2+c^2+d^2}-a_{11})\in \mathbb{C}^2.
\end{equation}
Similar to the discussion of the eigenvalues, the only singularity is at $r=0$, which becomes a smooth function on
$[S_0,\{0\}]$. It follows that $v_{+}$ and $v_-$ span two smooth line
bundles on $[S_0,\{0\}]$.

If we do the projective blow up instead, which identifies the
antipodal points in the front face of $\bbS^2$ to get $\mathbb{RP}^2$, namely
\begin{equation}
\hat S_{0,p}=\{(x,l)|x\in l\} \subset \bbR^3 \times \bbR\bbP^2
\end{equation}
which we can cover with three coordinate patches: 
$$
(x_1,y_1,z_1)=(c, \frac{d}{c}, \frac{a_{11}}{c})\in \bbR^3 
$$
and the other two $(x_2,y_2,z_2), (x_3,y_3,z_3)$=$(d,\frac{c}{d}, \frac{a_{11}}{d}), (a_{11}, \frac{c}{a_{11}}, \frac{d}{a_{11}})$ are similar.
The two eigenvalues we get from here are
$$
v_\pm=\pm\sqrt{a_{11}^2+c^2+d^2}=\pm|x_1|\sqrt{(1+y_1^2+z_1^2)}.
$$
which is smooth across $\{x_1=0\}$. Similar discussions hold for the other two coordinate patches.

However, the
trivial bundle does not decompose into two line bundles as in the
radial case. The nontriviality of eigenbundles can be seen by taking a homotopically nontrivial loop in $\mathbb{RP}^2$
$$
l=\beta^{-1}(\{r=1\}) \subset \hat S_{0,p}.
$$ 
This curve intersects the line $c=d=0$ twice, which hits at two different places thus both $a_{11}^{\pm}=\pm 1$ are on the curve, and (\ref{eigenvecor}) shows that starting from $v_-=(0, -2)=(0, -2a_{11}^{+})$, this turns into
$v_+=(0,-2)=(0,2a_{11}^-)$, which means the two eigenvectors are not separated by projective blow up.

Now that we have done the radial resolution for the trace free slice $S_0$, the resolution of $S$ follows. Consider $S$ as a
3-dim vector bundle on $\mathbb{R}$ with trace being the projection map, then at each base point $\lambda$, the fiber is $S_0+\lambda I$. The resolution is $[S_0+\lambda I;\lambda I]\cong [S_0;\{0\}]$. Since the trace direction is transversal to the blow up, and therefore
\begin{equation}
[S; \bbR I]=[S_0;\{0\}] \times \bbR.
\end{equation}  
And because the trace doesn't change the eigenvectors, the smoothness
follows.
\end{proof}

To proceed to higher dimensions, we first discuss the partition of eigenvalues into clusters. The basic case is when the eigenvalues are divided into two clusters, then the $U(n)$ action of the matrices can be decomposed to two commuting actions. 

\begin{definition}[spectral gap]\label{gap}
A connected neighborhood $U \subset S$ has a spectral gap
at $c\in \bbR$, if $c$ is not an
eigenvalue of X, for any $X \in U$.
\end{definition}
Note here that since U is connected, the number of eigenvalues less than c stays the same for all $X \in U$, denoted by k.

\begin{lemma}[local eigenspace decomposition]\label{decom}
If a bounded neighborhood $U \subset S(n)$ has a spectral gap at $c$,
then the matrices in U can be decomposed into two smooth self-adjoint commuting
matrices: 
$$X=L_X+R_X, L_XR_X=R_XL_X.$$
with $\operatorname{rank(}L_X) = k$, $\operatorname{rank}(R_X)=n-k$.
\end{lemma}
\begin{proof}

Let $\gamma$ be a simple closed curve on $\mathbb{C}$ such that it
intersects with $\mathbb{R}$ only at $-R$ and $c$, where $R$ is a
sufficiently large number such that $-R$ is less than any eigenvalues of the matrices contained in $U$. In this way, for any matrix $X \in U$, the $k$ smallest eigenvalues are contained inside $\gamma$. We consider the operator
\[
P_X: \mathbb{C}^n \rightarrow \mathbb{C}^n
\]
\begin{equation}\label{equ:PX}
 P_X:=-\frac{1}{2\pi i}\oint_\gamma (X-sI)^{-1}ds.
\end{equation}
Since the resolvent is nonsingular on $\gamma$, $P_X$ is a
well-defined operator and varies smoothly with $X$, the integral
is independent of choice of $\gamma$ up to homotopy.

First we show that $P_X$ is a projection operator, i.e.
\begin{equation}\label{projection}
 P_X^2=P_X.
\end{equation}
Let $\gamma_s$ and $\gamma_t$ be two curves satisfying the above
condition with $\gamma_s$ completely inside $\gamma_t$, then
\begin{equation*}
 \begin{array}{rl}
P_X^2&=-\frac{1}{4\pi^2}\oint_{\gamma_t} (X-tI)^{-1}dt (\oint_{\gamma_s} (X-sI)^{-1}ds)\\
&=-\frac{1}{4\pi^2}\oint_{\gamma_t} dt [\oint_{\gamma_s}
\frac{1}{s-t}(X-sI)^{-1}ds-\oint_{\gamma_s}
\frac{1}{s-t}(X-tI)^{-1}ds]\\
&=I-II
\end{array}
\end{equation*}
where using the fact that $s$ is completely inside $\gamma_t$
\begin{equation*}
I=-\frac{1}{4\pi^2}\oint_{\gamma_s} \frac{1}{X-sI}ds \oint_{\gamma_t}
\frac{1}{s-t}dt=-\frac{1}{4\pi^2}(-2\pi i) \oint_{\gamma_s}
\frac{1}{X-sI}ds=P_X
\end{equation*}
and any t on $\gamma_t$ is outside of the loop $\gamma_s$
\[
\oint_{\gamma_s} \frac{1}{s-t}ds=0
\]
we have 
$$
II= -\frac{1}{4\pi^2}\oint_{\gamma_t} (X-tI)^{-1} dt \oint_{\gamma_s}\frac{1}{s-t} ds =0.
$$ 
This proves~\eqref{projection}.

Then we show that $P_X$ is self-adjoint. This is because
\begin{equation*}
P_X^*=\frac{1}{2\pi i}\int_{\gamma}((X-sI)^{-1})^*d\bar
s=\frac{1}{2\pi i}\int_{-\bar \gamma}(X-sI)ds=P_X.
\end{equation*}

$P_X$ maps $\mathbb{R}^n$ to the invariant subspace spanned by
the eigenvectors corresponding to eigenvalues that are less than $c$.
We denote this invariant subspace by $L$ and its orthogonal complement by $R$.  Write X as the diagonalization 
$X = V\Lambda V^{-1}$ where $\Lambda$ is the eigenvalue matrix and $V$ is the matrix whose columns are the eigenvectors of $X$. Then $L$ is spanned by the first $k$ columns of $V$. 
Take one of the eigenvectors $v_j \in L, j= 1,2, ..., k$,
\begin{equation*}
 P_X v_j=-\frac{1}{2\pi i}\oint_\gamma (X-sI)^{-1}v_j ds= -\frac{1}{2\pi i} \oint  V(\Lambda-sI)^{-1} V^{-1}v_jds=-\frac{1}{2\pi i} v_j \oint \frac{1}{\lambda_j-s}ds= v_j.
\end{equation*}
Similarly for $v_j \in R$ that corresponds to an eigenvalue greater than c (therefore $\lambda_j$ is outside the loop),
\begin{equation*}
 P_X v_j=-\frac{1}{2\pi i} v_j \oint \frac{1}{\lambda_j-s}ds= 0, 
\end{equation*}
therefore
$$
(I-P_X)v_j=v_j, \forall v_j \in R.
$$

Then using the projection $P_X$ we define two operators $L_X$ and $R_X$ as
\begin{equation}
 L_X:=P_X X P_X
\end{equation}
and
\begin{equation}
 R_X:=(I-P_X)X(I-P_X)
\end{equation}
Since $P_X$ is smooth, the two operators are also smooth. Moreover,
using the fact that $P_X$ is a projection onto the invariant subspace $L$, we have
\[ (I-P_X)X P_X=P_X X(I-P_X)=0\]
therefore
\[X=L_X+R_X.\]

For an eigenvector $v\in L$,
\begin{equation}
  L_Xv =Xv, R_X v =0, 
\end{equation}
i.e. $L_X$ equals to $X$ when restricted to $L$, similarly $R_X|_R=X$. Since $P_X^*=P_X$, $L_X$ and
$R_X$ are also self-adjoint. In this way we get two commuting lower rank matrices $L_X$ and $R_X$.
\end{proof}

It is natural to have a finer decomposition when there is more than one spectral gap in the neighborhood, and we have the following corollary.
\begin{corollary}\label{finerdecom}
If the eigenvalues of matrices in a neighborhood U can be grouped into k clusters,
then the matrices can be decomposed into k lower rank self-adjoint commuting matrices smoothly.
\end{corollary}
\begin{proof}
Do the decomposition inductively. If $k=2$, then it is the case in lemma~\ref{decom}. Suppose the decomposition for $k=l-1$ is defined. Then for $k=l$, since the eigenvalues can also be divided into $2$ clusters (by combining the smallest $l-1$ groups of eigenvalues together), then $X=L_X+R_X$, with $L_X$ and $R_X$ corresponding to the two intervals. Then $L_X$ satisfies the separation condition for $l-1$ clusters, so by induction, $L_X=L_1+...+L_{l-1}$. Therefore, $X=L_1+L_2+...+L_{l-1} + R_X$ is the desired division.
\end{proof}

Using lemma~\ref{decom} of decomposition of matrices in a neighborhood, we can now show that locally the trivial bundle $S\times \bbC^n$ decomposes into two subspaces if there is a spectral gap. Moreover, locally there is a product structure of two lower dimensional matrices. In order to see this, we need to introduce the Grassmanian. 
Let $Gr_\mathbb{C}(n,k)$ denote the Grassmannian, i.e. the set of
$k$-dim subspaces in $\mathbb{C}^n$.
Consider the tautological vector bundle over Grassmanian:
\[\pi_{k}: T_k
\rightarrow Gr_{\mathbb{C}}(n,k), \pi^{-1}(p)=V(p).\] where each
fibre is a k-dimensional subspace in $\mathbb{C}^n$, with
self-adjoint operators acting on it. Similarly, we define $T_{n-k}$
to be the orthogonal complement of $T_k$:
\[
\pi_{n-k}: T_{n-k} \rightarrow Gr_{\mathbb{C}}(n,k),
\pi^{-1}(p)=V(p)^\perp.
\]

\begin{definition}[Operator bundle]\label{Pk}
Let $P_k$ (resp. $P_{n-k}$) be the bundles over $Gr_{\bbC}(n,k)$ of the fibre-wise
self-adjoint operators on the tautological bundle $T_k$ (resp. $T_{n-k}$).
\end{definition}

Take the Whitney sum of the two bundles
\begin{equation}
\pi: P_k\oplus P_{n-k}\rightarrow Gr_{\bbC}(n,k).
\end{equation}
 Each of its fiber can be identified with $S(k) \oplus S(n-k)$ when we pick a basis. There is a U(n)-action on this bundle:
\begin{multline}
g\cdot (p,(p_k,p_{n-k}))=(g\cdot p, (g\circ p_k \circ g^{-1},g\circ
p_{n-k} \circ g^{-1})),\\ p \in Gr_{\bbC}(n,k), p_k \in P_k(p), p_{n-k} \in P_{n-k}(p).
\end{multline}

Suppose an open neighborhood $U\in S$ satisfies the spectral gap
condition. Let $U(n)\cdot U$ be the group invariant neighborhood
generated by $U$, that is,
\begin{equation}
U(n) \cdot U:=\cup_{g \in U(n)}g\cdot U.
\end{equation}
Then $U(n)\cdot U$ is open and connected, and also satisfies the
spectral gap condition as $U$ does, since U(n) action preserves the
eigenvalues. From the proof of the lemma~\ref{decom}, it is shown that in the
neighborhood, the trivial $\mathbb{C}^n$ bundle over $U$ naturally splits
into two subbundles $E^{k} \oplus E^{n-k}$. And this gives a local
product structure. We will prove that, for a U(n)-invariant
neighborhood, there is actually a group equivariant homeomorphism with
the operator bundles defined above.

\begin{lemma}[bundle map]\label{product}
 If a point $X_0 \in S$ satisfies the spectral gap
condition, then there is a neighborhood $ V \subset S$ such that 
$V$ is homeomorphic to a neighborhood in the product of lower rank matrices and Grassmanian, i.e. 
$$
\phi: V \cong V(k)\times V(n-k)\times
V_{Gr} \subset S(k) \times S(n-k) \times Gr_\mathbb{C}(n,k),
$$
which is contained in $P_k\oplus P_{n-k}$ as defined in Definition~\ref{Pk}. Moreover, $U(n)\cdot V$ is homeomorphic to a neighborhood $W\subset P_k \oplus P_{n-k}$ such that $\pi(W)=Gr_{\bbC}(n,k)$ and the map $\phi$ is $U(n)$-equivariant. 
\end{lemma}
\begin{proof}

From the proof of Lemma \ref{decom}, there is a neighborhood $X_{0} \in U \in S$, such that each
element $X \in U$ is decomposed into $L_X+R_X$. Moreover, it
induces a decomposition of the trivial bundle $U \times
\mathbb{C}^n$ into two subbundles:
\begin{equation}\label{Ek}
U \times \mathbb{C}^n=E^{k} \oplus E^{n-k}
\end{equation}
where $E^{k}(X)$ and $E^{n-k}(X)$ are determined by the projection
operator $P_X$ defined in equation (\ref{equ:PX}):
\begin{equation}
E^{k}(X)=Im(P_X), E^{n-k}(X)=Im(P_X)^\perp
\end{equation}

Let $(\xi_1,...\xi_k)$ be the basis for $E^{k}(X_0)$. $E^{k}$ over $U$ is an
open neighborhood in $Gr_\mathbb{C}(n,k)$. We can find a neighborhood V of $X_0$ (possibly smaller than $U$) such that, for every
point in V, the k-dimensional space $E^{k}$ projects onto $E^{k}(X_0)$.
And an orthonormal basis of $E^{k}(X)$ is uniquely determined by
requiring the projection of the first $j$ vectors to $E^{k}(X_0)$ spans
$(\xi_1,...\xi_j)$ for every $j$ smaller than $k$.  In this way we find a basis for each fiber of $E^{k}$ and $E^{k}$ is
trivialized to be a k-dimensional vector bundle on V. Since the action of X
on $\mathbb{C}^n$ has been decomposed to $L_X$ and $R_X$, then with
the choice of basis, the action of $L_X$ on $E^{k}(X)$ gives a $k\times k$
self-adjoint matrix, and by continuity, these matrices form a
neighborhood $V_k$ in $S(k)$. And the same argument works for $R_X$.

Therefore, we have the following map $\phi$:
\begin{equation}
\begin{array}{c}
\phi: V \rightarrow  P_k\oplus P_{n-k}\\
X \mapsto (E^{k}(X), (L_X|_{E^{k}(X)},R_X|_{E^{n-k}(X)}))
\end{array}
\end{equation}
We show this map is a homeomorphism between $V$ and $\phi(V)$. 
It is injective since the actions of the two invariant subspaces
uniquely determine the action on $\mathbb{C}^n$, therefore gives
the unique operator $X$. Surjectivity is easy to see. The continuity of $\phi$ and $\phi^{-1}$
comes from the continuity of the projection operator defined in Lemma~\ref{blowup}. 

Now take $U(n)\cdot V$, since $E^{k}$ takes every possible k-subspace of $\bbC^n$ under the action of $U(n)$, we know that the first entry of $\phi(U(n)\cdot V)$ maps onto $Gr_\bbC(n,k)$. Moreover, since the decomposition respects
the action of $U(n)$, it is easily seen that, for $g\in U(n), X\in
U(n) \cdot V$,
\begin{equation}
\phi(g\cdot X)=(g\cdot E^{k}(X), (g\circ L_X \circ g^{-1},g\circ R_X
\circ g^{-1}))=g\cdot (\phi(X))
\end{equation}
which means the map is U(n)-equivariant.
\end{proof}

To do the induction, we will need to define an index on the inclusion of isotropy types, so the blow up procedure could be done in the partial order given by the index. Recall that two matrices have the same isotropy type if they have the same ``clustering'' of eigenvalues. Now we define the isotropy index of a matrix $X$ as follows.
\begin{definition}[Isotropy index]
Suppose the eigenvalues of a matrix $X$ are 
$$
\lambda_1=..=\lambda_{i_1} < \lambda_{i_1+1}=..=\lambda_{i_2} < \lambda_{i_2+1}=...<\lambda_{i_{k-1}+1}=...=\lambda_{n}
$$
then the isotropy index of $X$ is defined as the set 
$$
I(X)=\{i_0=0, i_1,i_2,...,i_{k-1}, i_{k}=n\}.
$$
We denote the set of all matrices with the same isotropy index $I$ as $S^{I}$.
\end{definition}
There is a partial order of this index on $S$ given by the inclusion. That is, if for two matrices $X$ and $Y$ we have $I(X) \subset I(Y)$, then we say that the order is $X\leq Y$. Note there is an inverse inclusion for isotropy groups. The smallest isotropy index is $I(\lambda I)=\{0,n\}$ while the isotropy group is $U(n)$ which is the largest. And the largest index is $\{0,1,2, ...,n-1,n\}$ which corresponds to $n$ distinct eigenvalues, and the isotropy group is the product of $n$ copies of $U(1)$.  

\begin{remark}\label{trans}
Except the most singular stratum $\{\lambda I\}$, the stratum of other isotropy types are not closed. In fact, the closure of a stratum $S^{I}$ will include all the stratum $S^{I'}$ with $I'\subset I$. However, the two sets $\{\lambda_{i_{1}}=\lambda_{i_{2}}=\dots= \lambda_{i_{k}}\}$ and $\{\lambda_{j_{1}}=\dots=\lambda_{j_{l}}\}$ are transversal once the set 
$\{\lambda_{\min\{i_{1}, j_{1}\}}=\dots=\lambda_{\max\{i_{k}, j_{l}\}}\}$ is blown up. So one can get $\hat S_{s}$ by blowing up these singular stratum by order of strict inclusion. However, in order to globally decompose the eigenbundle, one needs to blow up all the intersections first as in $\hat S_{r}$ (the proof is given later).
\end{remark}

For $\hat S_{r}$ and $\hat S_{p}$, the total blow up of $S(n)$ is done by iteratively blowing up the singular strata by the order of isotropy indices. The first step is to blow up the most singular stratum $S^{\{0,n\}}=\{\bbR I\}$:
$$
[S(n);S^{\{0,n\}}].
$$
 After that we blow up the second smallest strata $S^{\{0,i,n\}}, i=1, \dots, n-1$. From the discussion above we know that, for any of such two strata, the intersection of their closure is exactly $S^{\{0,n\}}$ which has been blown up. Therefore one can blow up these $S^{\{0,i,n\}}$ in any order:
 $$
 [S(n);S^{\{0,n\}}; \cup_{i=1}^{n-1}S^{\{0,i,n\}}].
 $$
 After the second step, the intersection of any two $S^{\{0,i,j,n\}}$ has been blown up. Therefore one can proceed by blowing up those strata in any order. Iteratively, one obtain the following space:
 \begin{equation}\label{final}
  [S(n);S^{\{0,n\}}; \cup_{i=1}^{n-1}S^{\{0,i,n\}}; \cup_{i,j}S^{\{0,i,j,n\}}; 
 \dots; \cup_{0\leq i_{1}<i_{2}<\dots < i_{n-2}\leq n} S^{\{0,i_{1}, \dots, i_{n-2},n\}}].
\end{equation}

In order to do the inductive proof to show this yields the full eigenresolution, the last lemma we need is the compatibility of
conjugacy class inclusion and the decomposition to two submatrices,
which shows the order of resolution is
compatible with the decomposition.
\begin{lemma}[Compatibility with conjugacy class]\label{order}
The partial order of conjugacy class inclusion is compatible with
the decomposition in lemma \ref{decom}.
\end{lemma}
\begin{proof}
Suppose a neighborhood $V \subset S(n)$ has a decomposition as lemma \ref{decom}.
We need to show that, if $S^{I}$ is the stratum of minimal
isotropy type in V, then this stratum
corresponds to the minimal isotropy type in $U(k)$ and $U(n-k)$.

Since $V$ satisfies the spectral gap condition, the isotropy groups for any elements in
 $V$ would be subgroups of  $U(k) \oplus U(n-k)$.
Suppose the minimal stratum corresponds to the index $I=\{0, i_1,...,i_m\}$ which must contain $k$ as one element because of the spectral gap condition. Then the isotropy type of two subgroups are $\{0, i_1, ..., k\}$ and $\{i_j-k=0, i_{j+1}-k, ..., n-k \}$. 
They would still be the minimal in each subgroup, otherwise when the two smallest elements are combined it will give a smaller index than $I$, which is a contradiction.
\end{proof}

Now we can finally prove theorem~\ref{blowup} using the above lemmas. 
\begin{proof}[Proof of Theorem \ref{blowup}] 
 We prove the theorem by induction of the matrix size. Except special remarks, the discussion below about $\hat S$ applies to all three kinds of resolutions. The $2\times 2$ case is shown in lemma~\ref{2dim}. Suppose the claim holds for all the cases up to $n-1$ dimensions.
Now we claim that, by an iterative blow up, we can get $\hat S(n)$ with eigenvalues and eigenbundles lifted to satisfy the eigenresolution properties. 

As in the $2 \times 2$ example, we shall first consider the trace free slice $S_0(n)$ since other slices have the same behavior in terms of smoothness of eigenvalues and eigenbundles, that is, $\hat S(n)=\hat S_{0}(n)\times \bbR$. Take the smallest index $I=\{0,n\}$ with the largest possible isotropy group $U(n)$, and the stratum in
$S_0(n)$ with such an isotropy group is a single point, the zero matrix. After
blowing up, we get $[S_0;\{0\}]$ as the first step. And in the total $S(n)$ space, this step corresponds to $[S; S^{\{0,n\}}=\{\bbR I\}]=[S_0;\{0\}] \times \bbR$.

For any other point $X\notin \{\bbR I\}$, one can find a bounded neighborhood $W$ such that the matrices in $W$ have a spectral gap as defined in Definition~\ref{gap}. Assume the first $k$ eigenvalues are uniformly bounded below $c$, then by Lemma~\ref{product} there is a fibration structure 
\begin{equation}\label{fiber}
\xymatrix{
V(k) \times V(n-k) \ar[r] &W \ar[d]^{\pi}\\
& Gr_{\bbC}(n,k)
}.
\end{equation}
And the trivial bundle $W\times \bbC^{n}$ naturally splits to the sum $E^{k} \oplus E^{n-k}$ as in~\eqref{Ek}.  
Because of the spectral gap, there are two smallest strata of type $\{\lambda_{i_{1}}=\dots \lambda_{i_{j}}\}$ and $\{\lambda_{i_{1}'}=\dots \lambda_{i_{j}'}\}$, with $i_{j}\leq k$ and $i_{1}'\geq k+1$, therefore the two strata are transversal as discussed in the remark~\ref{trans}, and can be blown up at the same time. This give the iteration step for $\hat S_{s}$.

Now we consider the radial and projective resolution.
For each fiber of $\pi$ in~\eqref{fiber}, consider the resolved space $\hat V(k) \times \hat V(n-k) \subset \hat S(k)\times \hat S(n-k)$, where the resolution is done by blowing up all the singular stratum \emph{inside} $V(k)$ and $V(n-k)$.  By induction the resolution $\hat V(k)$ resolves the singularity for the first $k$ eigenvalues, and $\hat V(n-k)$ resolves 
the rest $n-k$ eigenvalues. For example, take a point $X \in S(5)$ with eigenvalues $\{\lambda_{1}=\lambda_{2} < \lambda_{3}=\lambda_{4}=\lambda_{5}\}$. Near this point there is a product decomposition $V(2) \times V(3) \times Gr_{\bbC}(5,2)$. After the resolution, $\hat V(2) \times \hat V(3)$ resolves the isotropy type $\left(\{0,2\} \cup \{0,1,2\}\right) \times \left(\{0,3\} \cup \{0,1,3\} \cup \{0,2,3\} \cup \{0,1,2,3\} \right)$, which, after adjusting numbering of eigenvalues, include all the isotropy types that could occur with this spectral gap in $W$. Let $\hat W$ be the this resolved space and denote the blow down map as
$$
\xymatrix{
\beta: \hat W \ar[rd]^{\hat \pi} \ar[r] & W \ar[d]^{\pi}\\
&Gr_{\bbC}(n,k)
}.
$$

Consider the two subbundles $E^{k}$ and $E^{n-k}$ under the pull back map from $\beta$:
\begin{equation}
\xymatrix{
\hat E^{k} \oplus \hat E^{n-k}  \ar[r]^{\beta} \ar[d]^{\hat \phi} & E^{k} \oplus E^{n-k} \ar[d]^{\phi}\\
\hat W \ar[r]^{\beta} & W
}.
\end{equation}
By induction assumptions $\hat V(k)$ and $\hat V(n-k)$ are eigenresolutions, hence $\hat E^{k}$ splits into line bundles $\oplus_{i=1}^{k} E_{i}$ over $\hat V(k)$ and same for $\hat E^{n-k}=\oplus_{i=k+1}^{n}E_{i}$ over $\hat V(n-k)$. With the local product structure of $\pi$, the Whitney sum $\hat E^{k} \oplus \hat E^{n-k}$ splits into $n$ eigenbundles locally. 

For the radial resolution $\hat S_{r}$, since the local product structure is $U(n)$-equivariant, extending to $\oplus_{i=1}^{n} U(n) \cdot E_{i}$ we get that the splitting of eigenbundles are global over $\hat W$. We have already shown in Lemma~\ref{2dim} that the projective resolution does not give a global eigendecomposition. Similarly, for the small resolution $\hat S_{s}$, one can find a closed curve in the base such that one eigenvector switches to another around the curve. We prove this by giving an example: consider the curve of 4 by 4 matrices of the form $X(t)=U(t)\Lambda(t)U(t)^{-1}$, $0\leq t \leq 1$, where 
$U(t)$ is unitary for all $t$, switching from the identity to its column permutation,
$$
U(t)=
\left\{
\begin{array}{ll}
(\vec e_{1}, \vec e_{2}, \vec e_{3}, \vec e_{4}) & 0\leq t \leq 1/3\\
U(t) & 1/3 \leq t \leq 2/3\\
(\vec e_{3}, \vec e_{4}, \vec e_{1}, \vec e_{2}) & 2/3 \leq t \leq 1
\end{array}
\right.$$
which smoothly permutates the eigenspace decomposition.
On the other hand, $\Lambda(t)$ is always diagonal, going through $\{\lambda_{1}=\lambda_{2}\}$ and $\{\lambda_{3}=\lambda_{4}\}$:
$$
\Lambda(t)=
\left\{
\begin{array}{ll}
\operatorname{diag}\{-1,-1, 1,1\} & t=0\\
\operatorname{diag}\{-1,-1, 1-t, 1+t\} & 0\leq t \leq 1/3 \\
\operatorname{diag}\{-1-t, -1+t, \frac{1}{3}+t, \frac{5}{3}-t\} & 1/3 \leq t \leq 2/3 \\
\operatorname{diag}\{-2+t, -t,1,1\} & 2/3 \leq t \leq 1\\
\operatorname{diag}\{-1, -1, 1, 1\} & t=1
\end{array}
\right.
$$
With $X(t)$ defined above, one can see that $X(0)=X(1)$ in the stratum that is not blown up in $\hat S_{s}$. Now consider the lift of the curve to $\hat S_{s}$, which is still a closed curve. Now one can immediately see that as $t$ goes from 0 to 1, 
the eigenspace for the first two eigenvalues switch from $\{e_{1},e_{2}\}$ to $\{e_{3}, e_{4}\}$. So one cannot obtain a global  decomposition.

Even though the eigenbundles do not always split, the three resolutions all resolve eigenvalues. Since blow-down map $\beta$ is injective on a dense open set, the eigenvalues extends to the front face to be $n$ smooth functions $f_{i}$ on $\hat W$ and the splitting of eigendata extends to $\hat E^{n-k} \oplus \hat E^{n-k}$ from nearby such that 
\begin{equation*}
\beta(X) v_i=f_i(X)v_i, \forall \ v_i \in E_i(X),\forall \  X \in \hat W.
\end{equation*}

According to lemma \ref{order} the isotropy index order is preserved when decomposed into two subspaces. By induction, to obtain the global eigenresolution, we have iteratively blown up the strata according to isotropy indices to get $\hat S_{r}$ as in~\eqref{final}.
\end{proof}


\begin{thebibliography}{99}

\bibitem{MR2560748}
P.~Albin and R.~B. Melrose, \emph{Resolution of smooth group actions}, Contemp.
  Math \textbf{535} (2011), 1--26.

\bibitem{baum1985cohomologie}
P.~Baum, J.~Brylinski, and R.~MacPherson, \emph{Cohomologie
  {\'e}quivariante d{\'e}localis{\'e}e}, CR Acad. Sci. Paris S{\'e}r. I Math
  \textbf{300} (1985), no.~17, 605--608.

\bibitem{borel2006compactifications}
A.~Borel and L.~Ji, \emph{Compactifications of symmetric and locally symmetric
  spaces}, Springer, 2006.

\bibitem{davis1978smooth}
M.~Davis, \emph{Smooth g-manifolds as collections of fiber bundles},
  Pacific Journal of Mathematics \textbf{77} (1978), no.~2, 315--363.

\bibitem{duistermaat2000lie}
J.J. Duistermaat and J.~Kolk, \emph{{L}ie {G}roups}, Springer Science \&
  Business Media, 2000.

\bibitem{hassell1995analytic}
A.~Hassell, R.~Mazzeo, and R.~B. Melrose, \emph{Analytic surgery and the
  accumulation of eigenvalues}, Communications in Analysis and Geometry
  \textbf{3} (1995), 115--222.

\bibitem{kurdyka2008hyperbolic}
K.~Kurdyka and L.~Paunescu, \emph{Hyperbolic polynomials and
  multiparameter real-analytic perturbation theory}, Duke Mathematical Journal
  \textbf{141} (2008), no.~1, 123--149.

\bibitem{resolution}
R.~B. Melrose, \emph{Differential analysis on manifolds with corners},
  http://math.mit.edu/~rbm/daomwc5.ps.

\bibitem{rainer2013perturbation}
A.~Rainer, \emph{Perturbation theory for normal operators}, Transactions of
  the American Mathematical Society \textbf{365} (2013), no.~10, 5545--5577.

\bibitem{rauch1980perturbation}
J.~Rauch, \emph{Perturbation theory for eigenvalues and resonances of
  {S}chr{\"o}dinger {H}amiltonians}, Journal of Functional Analysis \textbf{35}
  (1980), no.~3, 304--315.

\bibitem{rellich1937storungstheorie}
F.~Rellich, \emph{St{\"o}rungstheorie der spektralzerlegung}, Mathematische
  Annalen \textbf{113} (1937), no.~1, 600--619.

\end{thebibliography}
\end{document}